\documentclass[11pt]{article}

\usepackage[margin=1.2in]{geometry}

\usepackage[T1]{fontenc}
\usepackage[utf8]{inputenc}
\usepackage[USenglish, UKenglish, english]{babel}
\usepackage{amssymb}
\usepackage{amsthm}
\usepackage{amsmath}
\usepackage[mathscr]{euscript}
\usepackage{enumitem}
\usepackage{graphicx}
\usepackage{MnSymbol}
 \let\mathscr\relax
\usepackage[scr]{rsfso}
\usepackage{tikz-cd}
\usepackage{tikz}
\usetikzlibrary{matrix,arrows,decorations.pathmorphing}
\usepackage{hyperref}
\usepackage{marginnote}
\usetikzlibrary{arrows.meta}

\usepackage{cases}

\usepackage{libertine}

\theoremstyle{definition}
\newtheorem{defin}{Definition}[section]
\theoremstyle{definition}
\newtheorem{ex}[defin]{Example}
\theoremstyle{plain}
\newtheorem{theo}[defin]{Theorem}
\theoremstyle{plain}
\newtheorem{prop}[defin]{Proposition}
\theoremstyle{plain}
\newtheorem{lem}[defin]{Lemma}
\theoremstyle{plain}

\theoremstyle{definition}
\newtheorem{rmk}[defin]{Remark}
\theoremstyle{definition}

\theoremstyle{definition}

\theoremstyle{plain}

\theoremstyle{definition}

\theoremstyle{definition}

\theoremstyle{plain}

\theoremstyle{plain}

\theoremstyle{plain}

\theoremstyle{definition}

\theoremstyle{plain}

\theoremstyle{definition}

\theoremstyle{definition}
\newtheorem*{defin*}{Definition}
\theoremstyle{definition}
\newtheorem*{ex*}{Example}
\theoremstyle{plain}
\newtheorem*{theo*}{Theorem}
\theoremstyle{plain}
\theoremstyle{plain}
\newtheorem*{conj*}{Conjecture}
\newtheorem*{prop*}{Proposition}
\theoremstyle{plain}
\newtheorem*{lem*}{Lemma}
\theoremstyle{plain}
\newtheorem*{cor*}{Corollary}
\theoremstyle{definition}
\newtheorem*{rmk*}{Remark}
\theoremstyle{definition}
\newtheorem*{exe*}{Exercise}

\theoremstyle{plain}
\newtheorem{theoA}{Theorem}

\theoremstyle{plain}

\theoremstyle{plain}

\theoremstyle{plain}

\theoremstyle{plain}

\numberwithin{equation}{section}

\usepackage{parskip}

\makeatletter
\def\thm@space@setup{%
  \thm@preskip=\parskip \thm@postskip=0pt
}
\makeatother

\setlist[enumerate]{label=(\roman*)}

\def\R{{\mathbf{R}}} 

\def\Br{{\rm Br}}

\def\ab{{\rm Ab}}

\def\Br{{\rm Br}}

\def\n{{\mathbf{N}}} 

\def\c{{\mathbf{C}}} 
 
\def\z{{\mathbf{Z}}} 

\def\O{{\mathbf{O}}} 

\def\u{{\bf{U}}}

\def\G{{\mathbf{G}}} 

\def\H{{\mathbf{H}}} 

\def\K{{\mathbf{K}}} 

\def\L{{\mathbf{L}}} 

\def\T{{\mathbf{T}}} 

\def\B{{\mathbf{B}}} 

\def\CL{{\mathcal{L}}} 

\makeatletter
\newcommand{\uset}[3][0ex]{%
  \mathrel{\mathop{#3}\limits_{
    \vbox to#1{\kern-7\ex@
    \hbox{$\scriptstyle#2$}\vss}}}}
\makeatother

\newcommand{\wh}[1]{\widehat{#1}}

\usepackage{sectsty}
\allsectionsfont{\centering}

\usepackage{xparse,etoolbox}

\newcommand{\blocktheorem}[1]{%
  \csletcs{old#1}{#1}
  \csletcs{endold#1}{end#1}
  \RenewDocumentEnvironment{#1}{o}
    {\par\addvspace{1.5ex}
     \noindent\begin{minipage}{\textwidth}
     \IfNoValueTF{##1}
       {\csuse{old#1}}
       {\csuse{old#1}[##1]}}
    {\csuse{endold#1}
     \end{minipage}
     \par\addvspace{1.5ex}}
}

\blocktheorem{theoA}
\blocktheorem{conjA}
\blocktheorem{paraA}
\blocktheorem{corA}

\makeatletter
\def\blfootnote{\gdef\@thefnmark{}\@footnotetext}
\makeatother

\title{
{\huge\bf The simplicial complex of Brauer pairs of a finite reductive group}\\
\author{\Large Damiano Rossi}
\date{}
\blfootnote{\emph{$2010$ Mathematical Subject Classification:} $20$J$05$, $20$G$40$, $20$C$20$, $55$P$91$
\\
\emph{Key words and phrases:} Simplicial complexes, Brauer pairs, finite reductive groups, generalised Harish-Chandra theory.
\\
This work was carried out during a stay of the author at the Mathematisches Forschungsinstitut Oberwolfach funded by a Leibniz Fellowship. The author would like to thank Marc Cabanes and Gunter Malle for useful comments on an earlier version of this paper.
}}

\begin{document}

\renewcommand{\thetheoA}{\Alph{theoA}}

\renewcommand{\thecorA}{\Alph{corA}}

\selectlanguage{english}

\maketitle

\begin{abstract}
In this paper we study the simplicial complex induced by the poset of Brauer pairs ordered by inclusion for the family of finite reductive groups. In the defining characteristic case the homotopy type of this simplicial complex coincides with that of the Tits building thanks to a well-known result of Quillen. On the other hand, in the non-defining characteristic case, we show that the simplicial complex of Brauer pairs is homotopy equivalent to a simplicial complex determined by generalised Harish-Chandra theory. This extends earlier results of the author on the Brown complex and makes use of the theory of connected subpairs and twisted block induction developed by Cabanes and Enguehard.
\end{abstract}

\section*{Introduction}

The poset $\mathcal{S}_\ell^\star(G)$ of non-trivial $\ell$-subgroups of a finite group $G$, with respect to a prime $\ell$ dividing the order of $G$, gives rise to a simplicial complex $\Delta(\mathcal{S}_\ell^\star(G))$ known as the Brown complex. This simplicial complex was first introduced by Brown in \cite{Bro75} and its homotopy properties were later described by Quillen in \cite{Qui78}. In particular, for $G=\G^F$ a finite reductive group in characteristic $p$, Quillen showed that the Brown complex $\Delta(\mathcal{S}_p^\star(\G^F))$ is homotopy equivalent to the Tits building of $\G^F$. More recently, the author considered the remaining non-defining characteristic case, where $p\neq \ell$, and showed in \cite{Ros-Homotopy} that the homotopy type of $\Delta(\mathcal{S}_\ell^\star(\G^F))$ can be described in terms of the generic Sylow theory developed by Brou\'e and Malle \cite{Bro-Mal92}. This result was then used to obtain a connection between the so-called local-global conjectures in group representation theory and certain statements in generalised Harish-Chandra theory introduced in \cite{Ros-Generalized_HC_theory_for_Dade} and \cite{Ros-Unip}.

In this paper we extend the above-mentioned homotopy equivalences to a representation theoretic setting by replacing $\ell$-subgroups with $\ell$-Brauer pairs. More precisely, let $B$ be a Brauer $\ell$-block of a finite group $G$ and denote by $\mathcal{S}_\ell^\star(B)$ the poset of non-trivial $B$-Brauer pairs ordered by inclusion as defined in Section \ref{sec:Preliminaries}. Observe that $\mathcal{S}_\ell^\star(B)$ is non-empty if and only if $B$ has positive defect and that the associated simplicial complex $\Delta(\mathcal{S}_\ell^\star(B))$ yields a natural generalisation of the Brown complex: if $B_0(G)$ is the principal $\ell$-block of $G$, there exists a natural homeomorphism between $\Delta(\mathcal{S}_\ell^\star(B_0(G)))$ and the Brown complex $\Delta(\mathcal{S}_\ell^\star(G))$. 

We now focus on the case where $G=\G^F$ is a finite reductive group in characteristic $p$ as before. When $p=\ell$, in analogy with Quillen's result, it turns out that for each $p$-block $B$ with positive defect the simplicial complex $\Delta(\mathcal{S}_p^\star(B))$ is homotopy equivalent to the Tits building. This is shown in Theorem \ref{thm:Equivalence for defining characteristic} below. More interesting is the case $p\neq \ell$. In this situation, for a Brauer $\ell$-block $B$, we use the theory of connected subpairs and twisted block induction introduced by Cabanes and Enguehard \cite{Cab-Eng99} to construct a poset $\CL_{e_\ell(q)}^\star(B)$ consisting of pairs $(\L,b_\L)$ where $\L$ is an $e_\ell(q)$-split Levi subgroup of $(\G,F)$, with $e_\ell(q)$ the multiplicative order of $q$ modulo $\ell$, and $b_\L$ is a Brauer $\ell$-block of $\L^F$. We refer the reader to Section \ref{sec:Non-defining characteristic} for a precise definition. The advantage of the simplicial complex $\Delta(\CL_{e_\ell(q)}^\star(B))$ is that it can be determined by generalised Harish-Chandra theory. We can describe the homotopy type of $\Delta(\mathcal{S}_\ell^\star(B))$ as follows. 

\begin{theoA}
\label{thm:Main, Homotopy equivalence in non-defining characteristic}
Let $\G$ be a connected reductive group defined over an algebraically closed field of characteristic $p$ and consider $F:\G\to \G$ a Frobenius endomorphism endowing the algebraic variety $\G$ with an $\mathbb{F}_q$-rational structure. Suppose that $\ell\in\pi(\G,F)$ as defined in Definition \ref{def:Prime condition} and that $\ell$ does not divide the order of $\z(\G)^F$. Then there exists a $\G^F$-homotopy equivalence 
\[\Delta\left(\mathcal{S}_\ell^\star(B)\right)\simeq\Delta\left(\CL_{e_\ell(q)}^\star(B)\right)\]
for every Brauer $\ell$-block $B$ of $\G^F$ with non-trivial defect and where $e_\ell(q)$ denotes the multiplicative order of $q$ modulo $\ell$.
\end{theoA}

The paper is structured as follows. In Section \ref{sec:Preliminaries} we collect some preliminary results on Brauer pairs. In particular, we introduce the notion of almost-centric Brauer pair and show that these pairs control the homotopy type of the simplicial complex $\Delta(\mathcal{S}_\ell^\star(B))$. In Section \ref{sec:Defining characteristic} we consider the case $G=\G^F$ and $p=\ell$, and prove the connection between the Tits building of $\G^F$ and the simplicial complex $\Delta(\mathcal{S}_p^\star(B))$. Finally, in Section \ref{sec:Non-defining characteristic} we consider the case $p\neq \ell$ and prove Theorem \ref{thm:Main, Homotopy equivalence in non-defining characteristic}. We then conclude the paper by showing how \cite[Theorem A]{Ros-Homotopy} can be recovered from our Theorem \ref{thm:Main, Homotopy equivalence in non-defining characteristic}.

\section{Preliminary results on Brauer pairs}
\label{sec:Preliminaries}

Let $G$ be a finite group and fix a prime number $\ell$. A \textit{Brauer pair} of $G$, with respect to the prime $\ell$, is a pair $(Q,b_Q)$ where $Q$ is an $\ell$-subgroup of $G$ and $b_Q$ is a Brauer $\ell$-block of $\c_G(Q)$. Using the Brauer map $\Br_Q$ (see \cite[Theorem 5.4.1]{Lin18I}) we can define a partial order relation on the set of Brauer pairs. If $(P,b_P)$ and $(Q,b_Q)$ are Brauer pairs of $G$, then we write $(Q,b_Q)\leq (P,b_P)$ if $Q\leq P$ and there exists a primitive idempotent $i$ in the algbera $(kG)^P$ of $P$-fixed point, and where $k$ is a large enough field of characteristic $\ell$, such that $\Br_P(i)e_{b_P}\neq 0$ and $\Br_Q(i)e_{b_Q}\neq 0$ (see \cite[Definition 6.3.2]{Lin18II}). Here $e_{b_P}$ and $e_{b_Q}$ are the primitive central idempotents of $k\c_G(P)$ and $k\c_G(Q)$ corresponding to the Brauer $\ell$-blocks $b_P$ and $b_Q$ respectively. Given a Brauer pair $(Q,b_Q)$ there exists a unique Brauer $\ell$-block $B$ of $G$ such that $(1,B)\leq (Q,b_Q)$ in which case we say that $(Q,b_Q)$ is a \textit{$B$-Brauer pair}. Moreover, if $(Q,b_Q)\leq (P,b_P)$ then $(Q,b_Q)$ is a $B$-Brauer pair if and only if $(P,b_P)$ is a $B$-Brauer pair according to \cite[Proposition 6.3.6]{Lin18II}. We denote by $\mathcal{S}_\ell^\star(B)$ the poset of $B$-Brauer pairs $(Q,b_Q)$ satisfying $Q\neq 1$. Observe that the poset $\mathcal{S}_\ell^\star(B)$ is non-empty if and only if $B$ has non-trivial defect. In fact, maximal $B$-Brauer pairs are of the form $(D,b_D)$ where $D$ is a defect subgroup of $B$ (see \cite[Theorem 6.3.7]{Lin18II}). Moreover, since $B$ is stable under the action of $G$ by conjugation, notice that $\mathcal{S}_\ell^\star(B)$ is actually a $G$-poset. We denote by $\n_G(Q,b_Q)$ the stabiliser of the Brauer pair $(Q,b_Q)$ under the action of $G$, that is the stabiliser of the block $b_Q$ in the normaliser $\n_G(Q)$. For further details on Brauer pairs we refer the reader to the original paper of Alperin and Brou\'e \cite{Alp-Bro79} and to Linckelmann's monograph \cite[Section 6.3]{Lin18II}.

Next, recall that for every poset $\mathcal{X}$ we can form a simplicial complex $\Delta(\mathcal{X})$ whose simplices are given by totally ordered chains of finite length in $\mathcal{X}$. In particular, if $\mathcal{S}_\ell^\star(G)$ is the set of non-trivial $\ell$-subgroups of $G$, then $\Delta(\mathcal{S}_\ell^\star(G))$ is the \textit{Brown complex} introduced in \cite{Bro75} and further studied in \cite{Qui78}. In the next lemma we show that the Brown complex of $G$ is homeomorphic to the simplicial complex $\Delta(\mathcal{S}_\ell^\star(B_0(G)))$ where $B_0(G)$ is the principal block of $G$.

\begin{lem}
\label{lem:Principal block and Brown complex}
Let $\ell$ be a prime dividing the order of $G$ and consider the Brown complex $\Delta(\mathcal{S}_\ell^\star(G))$. Then, there is a homeomorphism of $\Delta(\mathcal{S}_\ell^\star(G))$ with $\Delta(\mathcal{S}_\ell^\star(B_0(G)))$ induced by the poset isomorphism
\begin{align*}
\mathcal{S}_\ell^\star(G)&\to \mathcal{S}_\ell^\star(B_0(G))
\\
Q &\mapsto (Q,B_0(\c_G(Q))).
\end{align*}
\end{lem}

\begin{proof}
Let $Q$ be an $\ell$-subgroup of $G$ and suppose that $b_Q$ is a block of $\c_G(Q)$ such that the Brauer pair $(Q,b_Q)$ belongs to $\mathcal{S}_\ell^\star(B_0(G))$. By Brauer's Third Main Theorem (see \cite[Theorem 6.3.14]{Lin18II}) the block $b_Q$ must coincide with the principal block $B_0(\c_G(Q))$ of $\c_G(Q)$ and we conclude that the assignment $Q\mapsto(Q,B_0(\c_G(Q)))$ defines an isomorphism of posets. The map of simplicial complexes obtained by extending this assignment to each simplex of $\Delta(\mathcal{S}_\ell^\star(G))$ is then a homeomorphism. 
\end{proof}

In the rest of this section, we reduce the study of the homotopy type of the simplicial complex $\Delta(\mathcal{S}_\ell^\star(B))$ to that of certain subcomplexes whose properties reflect the behaviour of the connected subpairs of Cabanes and Enguehard (see \cite[Definition-Proposition 2.1]{Cab-Eng99}). These results will be used in Section \ref{sec:Non-defining characteristic} to prove Theorem \ref{thm:Main, Homotopy equivalence in non-defining characteristic}. First, we define the set $\ab_\ell^\star(B)$ consisting of those Brauer pairs $(Q,b_Q)$ belonging to $\mathcal{S}_\ell^\star(B)$ and with $Q$ abelian. The following result extends a well known property of the Brown complex to the simplicial complex of Brauer pairs.

\begin{lem}
\label{lem:Abelian Brauer pairs and homotopy equivalence}
Let $B$ be an $\ell$-block of the finite group $G$. Then the inclusion of $G$-posets $\iota:\ab_\ell^\star(B)\to\mathcal{S}_\ell^\star(B)$ induces a $G$-homotopy equivalence $\Delta(\iota):\Delta(\ab_\ell^\star(B))\to \Delta(\mathcal{S}_\ell^\star(B))$.
\end{lem}

\begin{proof}
Without loss of generality we may assume that $B$ has non-trivial defect, for otherwise both $\mathcal{S}_\ell^\star(B)$ and $\ab_\ell^\star(B)$ are empty. Now, fix an element $(Q,b_Q)$ of $\mathcal{S}_\ell^\star(B)$ and denote by $\mathcal{X}$ the set of pairs $(P,b_P)$ of $\ab_\ell^\star(B)$ satisfying $(P,b_P)\leq (Q,b_Q)$. By Quillen's Theorem A \cite[Proposition 1.6]{Qui78} (see also the statement given in \cite[Lemma 1.1]{Ros-Homotopy}) it is enough to show that the simplicial complex $\Delta(\mathcal{X})$ is $\n_G(Q,b_Q)$-contractible. If $c$ is the unique block of $\c_G(\z(Q))$ such that $(\z(Q),c)\leq (Q,b_Q)$ (see \cite[Theorem 6.3.3]{Lin18II}), then $(\z(Q),c)$ is $\n_G(Q,b_Q)$-invariant and we claim that $\Delta(\mathcal{X})$ is $\n_G(Q,b_Q)$-join contractible via $(\z(Q),c)$ (see \cite[Section 1]{Ros-Homotopy}). Let $(P,b_p)\in\mathcal{X}$. Since $P$ is contained in $Q$, we deduce that $\z(Q)$ centralises $P$ and therefore $P\z(Q)$ is a well-defined abelian $\ell$-subgroup contained in $Q$. By using \cite[Theorem 6.3.3]{Lin18II} once again, we can now find a unique block $c_P$ of $\c_G(P\z(Q))$ such that $(P\z(Q),c_P)\leq (Q,b_Q)$. Furthermore, the uniqueness part of \cite[Theorem 6.3.3]{Lin18II} implies that $(P\z(Q),c_P)$ is the join of $(P,b_P)$ and $(\z(Q),c)$ and therefore $\Delta(\mathcal{X})$ is $\n_G(Q,b_Q)$-contractible according to \cite[Corollary 1.3]{Ros-Homotopy}.
\end{proof}

The above lemma still holds if we replace $\ab_\ell^\star(B)$ with the subposet of Brauer pairs $(Q,b_Q)$ such that $Q$ is an elementary abelian $\ell$-subgroup. We include this observation in the following remark.

\begin{rmk}
By replacing $\z(Q)$ with the subgroup $\Omega_1(\z(Q))$ of elements of order $\ell$ in $\z(Q)$, the above argument shows that $\Delta(\mathcal{S}_\ell^\star(B))$ is $G$-homotopy equivalent to the simplicial complex associated to the subposet of Brauer pairs $(Q,b_Q)$ with $Q$ elementary abelian. This fact, together with Lemma \ref{lem:Principal block and Brown complex}, can be used to recover Quillen's lemma \cite[Lemma 2.2]{Qui78} on the Brown complex.
\end{rmk}

Recall that a Brauer pair $(Q,b_Q)$ is called \textit{centric} if $b_Q$ has defect $\z(Q)$ in $\c_G(Q)$ (see, for instance, \cite[Definition 5.6]{Cab18}). In this case, it follows from \cite[Theorem 4.8]{Nav98} that $\z(Q)=\z(\c_G(Q))_\ell$ and where for any finite abelian group $H$ we define $H_\ell:=\O_\ell(H)$. On the other hand, observe that if $(Q,b_Q)$ is a Brauer pair satisfying $\z(Q)=\z(\c_G(Q))_\ell$ then the defect groups of $b_Q$ in $\c_G(Q)$ might still be larger than $\z(Q)$.

\begin{ex}
\label{ex:Almost centric non centric}
Let $G$ be the direct product of the cyclic group $C_2$ with the symmetric group $S_3$ and consider $\ell=2$. Denote by $Q=\z(G)$ the centre of $G$ which is cyclic of order $2$, and consider the principal block $B_0$ of $G=\c_G(Q)$. In this case, $\z(Q)=Q=\z(\c_G(Q))_\ell$ while the defect groups of $B_0$ have order $4$. Therefore $(Q,B_0)$ is not centric but satisfies the equality $\z(Q)=\z(\c_G(Q))_\ell$.
\end{ex}

The above discussion leads to the following definition of almost-centric Brauer pairs.

\begin{defin}
An $\ell$-subgroup $Q$ of a finite group $G$ is called \textit{almost-centric} if $\z(Q)=\z(\c_G(Q))_\ell$. Furthermore, we say that a Brauer pair $(Q,b_Q)$ is \textit{almost-centric} if the $\ell$-subgroup $Q$ is almost-centric. We denote by $\mathcal{S}_\ell^\star(B)^{\rm ac}$ the subset of $\mathcal{S}_\ell^\star(B)$ consisting of almost-centric Brauer pairs and by $\ab_\ell^\star(B)^{\rm ac}$ its intersection with $\ab_\ell^\star(B)$. Observe that the action of $G$ by conjugation on $\mathcal{S}_\ell^\star(B)$ restricts to the subsets $\mathcal{S}_\ell^\star(B)^{\rm ac}$ and $\ab_\ell^\star(B)^{\rm ac}$.
\end{defin}

It follows from the above definition that $\mathcal{S}_\ell^\star(B)^{\rm ac}$ and $\ab_\ell^\star(B)^{\rm ac}$ are $G$-subposets of $\mathcal{S}_\ell^\star(B)$. We can then refine the statement of Lemma \ref{lem:Abelian Brauer pairs and homotopy equivalence} and show that the homotopy type of the simplicial complex $\Delta(\mathcal{S}_\ell^\star(B))$ is determined by the abelian almost-centric Brauer pairs in $\ab_\ell^\star(B)^{\rm ac}$. Observe that the proof below also shows that $\ab_\ell^\star(B)^{\rm ac}$ is non-empty whenever $B$ has non-trivial defect group.

\begin{prop}
\label{prop:Almost centric pairs and homotopy equivalence}
Let $B$ be an $\ell$-block of the finite group $G$. Then the inclusion of posets $\iota:\ab_\ell^\star(B)^{\rm ac}\to \mathcal{S}_\ell^\star(B)$ induces a $G$-homotopy equivalence $\Delta(\iota):\Delta(\ab_\ell^\star(B)^{\rm ac})\to \Delta(\mathcal{S}_\ell^\star(B))$.
\end{prop}

\begin{proof}
Without loss of generality we may assume that $B$ has non-trivial defect. Moreover, using Lemma \ref{lem:Abelian Brauer pairs and homotopy equivalence}, it suffices to show that the inclusion $\iota:\ab_\ell^\star(B)^{\rm ac}\to \ab_\ell^\star(B)$ induces a $G$-homotopy equivalence $\Delta(\iota):\Delta(\ab_\ell^\star(B)^{\rm ac})\to \Delta(\ab_\ell^\star(B))$. By Quillen's Theorem A \cite[Proposition 1.6]{Qui78} (we actually use the stronger form stated in \cite[Lemma 1.1 (ii)]{Ros-Homotopy}) it is enough to show that, given a Brauer pair $(Q,b_Q)\in\ab_\ell^\star(B)$, the simplicial complex $\Delta(\mathcal{X})$ is $\n_G(Q,b_Q)$-contractible where $\mathcal{X}$ denotes the poset of pairs $(P,b_P)\in\ab_\ell^\star(B)^{\rm ac}$ such that $(Q,b_Q)\leq (P,b_P)$. We claim that the pair $(\z(\c_G(Q))_\ell,b_Q)$ is an $\n_G(Q,b_Q)$-invariant minimum in the poset $\mathcal{X}$ from which we conclude that $\Delta(\mathcal{X})$ is $\n_G(Q,b_Q)$-contractible thanks to \cite[Corollary 1.3]{Ros-Homotopy}. First, notice that because $Q$ is abelian it is contained in $\z(\c_G(Q))_\ell$ and therefore that $\c_G(Q)=\c_\G(\z(\c_G(Q))_\ell)$ by elementary group theory. In particular, it follows that the Brauer pair $(\z(\c_G(Q))_\ell,b_Q)$ is well-defined and belongs to $\ab_\ell^\star(B)^{\rm ac}$. In addition, noticing that $\z(\c_\G(Q))_\ell$ is a characteristic subgroup of $\c_G(Q)$ and that $\c_G(Q)$ is normalised by $\n_G(Q)$, we have that $(\z(\c_G(Q))_\ell,b_Q)$ is invariant under the action of $\n_G(Q,b_Q)$. This shows that $(\z(\c_G(Q))_\ell,b_Q)$ is an $\n_G(Q,b_Q)$-invariant element of $\mathcal{X}$. Suppose now that $(P,b_P)$ is an almost-centric Brauer pair in $\mathcal{X}$. Since $Q\leq P$ and $P$ is abelian, we deduce that $P\leq \c_G(P)\leq \c_G(Q)$ and hence that $\z(\c_G(Q))\leq \c_G(P)$. Then $\z(\c_G(Q))_\ell\leq \z(\c_G(P))_\ell=\z(P)=P$ because $P$ is almost-centric and thus $(\z(\c_G(Q))_\ell,b_Q)\leq (P,b_P)$ by applying \cite[Theorem 6.3.3]{Lin18II} to the inclusions $Q\leq \z(\c_G(Q))_\ell\leq P$ and recalling that $(Q,b_Q)\leq (P,b_P)$. Therefore $(\z(\c_G(Q))_\ell,b_Q)$ is a minimum in the poset $\mathcal{X}$ as claimed previously and this completes the proof.
\end{proof}

We conclude this section with a remark on centric Brauer pairs. It follows from the proof of Proposition \ref{prop:Almost centric pairs and homotopy equivalence} that for every intermediate poset $\ab_\ell^\star(B)^{\rm ac}\subseteq\mathcal{P}\subseteq \mathcal{S}_\ell^\star(B)$ the inclusion $\iota:\mathcal{P}\to \mathcal{S}_\ell^\star(B)$ induces a homotopy equivalence of $\Delta(\mathcal{S}_\ell^\star(B))$ with $\Delta(\mathcal{P})$. In particular, $\mathcal{S}_\ell^\star(B)^{\rm ac}$ and $\mathcal{S}_\ell^\star(B)$ induce homotopy equivalent simplicial complexes. It is interesting to point out that, however, this result fails if we replace almost-centric Brauer pairs with centric Brauer pairs. The following example is due to Gelvin and M{\o}ller (see {\cite[Example 6.2]{Gel-Mol15}).

\begin{ex}
Consider $G=C_2\times S_3$, $\ell=2$ and $B_0$ the principal $2$-block of $G$ as in Example \ref{ex:Almost centric non centric}. Since $\O_2(G)\neq 1$, it follows from \cite[Proposition 2.1 and Proposition 2.4]{Qui78} and Lemma \ref{lem:Principal block and Brown complex} that $\Delta(\mathcal{S}_\ell^\star(B_0))$ is contractible. On the other hand the poset of centric $B_0$-Brauer pairs induces a discrete simplicial complex consisting of three zero dimensional simplices corresponding to the three Sylow $2$-subgroups of $G$.
\end{ex}

\section{The defining characteristic case $\ell=p$}
\label{sec:Defining characteristic}

Let $\G$ be a connected reductive group defined over an algebraically closed field $\mathbb{F}$ of prime characteristic $p$ and consider a Frobenius endomorphism $F:\G\to \G$ corresponding to an $\mathbb{F}_q$-rational structure on the algebraic variety $\G$ for a power $q$ of $p$. From now on we restrict our attention to the case where the finite group $G$ coincides with the finite reductive group $\G^F$ consisting of the $\mathbb{F}_q$-rational points in $\G$. In this section we assume that the defining characteristic $p$ of $\G^F$ coincides with the prime $\ell$ with respect to which Brauer blocks are defined. The non-defining characteristic case, i.e. the case $\ell\neq p$, will be considered in Section \ref{sec:Non-defining characteristic}. We start by recalling the following well-known identity.

\begin{lem}
\label{lem:Centraliser of unipotent radical}
Let $\B$ be an $F$-stable Borel subgroup of $\G$ with unipotent radical $\u$. Then $\c_{\G^F}(\u^F)=\z(\G^F)\z(\u^F)$.
\end{lem}

\begin{proof}
By \cite[Corollary 12.2.4 and Proposition 12.2.14]{Dig-Mic20} there exists an $F$-stable regular unipotent element $u$ of $\G^F$. Recall that $u$ is a $p$-element of $\G^F$ and that $\u^F$ is a Sylow $p$-subgroup of $\G^F$ according to \cite[Proposition 1.1.5 and Proposition 4.4.1]{Dig-Mic20}. We can therefore assume that $u$ belongs to $\u^F$. Now, \cite[Lemma 12.2.3]{Dig-Mic20} implies that $\c_{\G^F}(u)=\z(\G^F)\c_{\u^F}(u)$ and we deduce that
\begin{align*}
\c_{\G^F}(\u^F)&=\c_{\G^F}(u)\cap \c_{\G^F}(\u^F)
\\
&=\z(\G^F)\c_{\u^F}(u)\cap \c_{\G^F}(\u^F)
\\
&=\z(\G^F)\left(\c_{\u^F}(u)\cap \c_{\G^F}(\u^F)\right)
\\
&=\z(\G^F)\z(\u^F)
\end{align*}
where we used Dedekind's modular law. This completes the proof.
\end{proof}

Let $\mathcal{P}(\G,F)$ be the poset consisting of $F$-stable parabolic subgroups of $\G$ ordered by inclusion. In this paper we define the \textit{Tits building} $\mathcal{B}(\G,F)$ of the finite reductive group $(\G,F)$ to be the associated simplicial complex $\Delta(\mathcal{P}(\G,F)^{\rm op})$. Here, for every given poset $\mathcal{X}$, we denote by $\mathcal{X}^{\rm op}$ the opposite poset given by reverse inclusions. What we just defined is, more precisely, the barycentric subdivision of the Tits building which is homeomorphic to it. We refer the reader to \cite[Section 6.8]{Ben98} for further details. The Tits building was shown to be homotopy equivalent to the Brown complex $\Delta(\mathcal{S}_p^\star(\G^F))$ by Quillen in \cite[Proposition 2.1 and Theorem 3.1]{Qui78}. The aim of this section is to extend this result to the blockwise set-up considered in this paper and prove that the simplicial complex $\Delta(\mathcal{S}_p^\star(B))$ for a $p$-block $B$ of $\G^F$ with non-trivial defect is $\G^F$-homotopy equivalent to the Tits building $\mathcal{B}(\G,F)$. This is the content of the following theorem. We remark that the following statement is an extension of the homotopy equivalence for the Brown complex $\Delta(\mathcal{S}_p^\star(\G^F))$ thanks to Lemma \ref{lem:Principal block and Brown complex}.

\begin{theo}
\label{thm:Equivalence for defining characteristic}
Let $B$ be a $p$-block of $\G^F$ with non-trivial defect and where $p$ is the defining characteristic of $\G^F$. If $\G$ is simple and $\G^F$ is perfect, then the simplicial complex $\Delta(\mathcal{S}^\star_p(B))$ is $\G^F$-homotopy equivalent to the Tits building $\mathcal{B}(\G,F)$.
\end{theo}

\begin{proof}
By \cite[Theorem 3.1 and Proposition 2.1]{Qui78} we deduce that the Brown complex $\Delta(\mathcal{S}_p^\star(\G^F))$ is homotopy equivalent to the Tits building $\mathcal{B}(\G,F)$. To see that this is actually a $\G^F$-homotopy equivalence, notice that the poset of proper parabolic subgroups underlying the Tits building is isomorphic (via a $\G^F$-equivariant map) to the opposite of the poset of proper $p$-radical subgroups of $\G^F$ while the latter induces a simplicial complex that is $\G^F$-homotopy equivalent to the Brown complex $\Delta(\mathcal{S}_p^\star(\G^F))$ thanks to \cite[Theorem 2]{The-Web91}. Therefore, it suffices to show that $\Delta(\mathcal{S}_p^\star(B))$ is $\G^F$-homotopy equivalent to the Brown complex. Observe that this claim follows already from Lemma \ref{lem:Principal block and Brown complex} in the case that $B$ is the principal block. Suppose now that the block $B$ is not principal.

Since $B$ has non-trivial defect by assumption, \cite{Hum71} implies that $B$ has defect group $U=\u^F$ a Sylow $p$-subgroup of $\G^F$. Moreover, following the proof of \cite[Theorem 6.18]{Cab-Eng04} we can find an irreducible character $1\neq\zeta$ of $\z(\G^F)$ that parametrises the $p$-block $B$. More precisely, by Lemma \ref{lem:Centraliser of unipotent radical} we have $\c_{\G^F}(U)=\z(\G^F)\z(U)$ and therefore $\epsilon_\zeta:=|\z(\G^F)|^{-1}\sum_{z\in\z(\G^F)}\zeta(z)z^{-1}$ determines a primitive idempotent of the group algebra $\mathbb{F}\c_{\G^F}(U)$ where $\mathbb{F}$ is the algebraically closed field of characteristic $p$ over which $\G$ is defined. Here, notice that the order $\z(\G^F)$ is invertible inside $\mathbb{F}$ since it consists of semisimple elements and applying \cite[Proposition 1.1.5]{Dig-Mic20}. Moreover, if we denote by $B_U$ the corresponding $p$-block of $\c_{\G^F}(U)$, then we have the inclusion of Brauer pairs $(1,B)\leq (U,B_U)$ and therefore $(U,B_U)$ belongs to $\mathcal{S}_p^\star(B)$. Now, \cite[Theorem 6.3.3]{Lin18II} shows that for every non-trivial $p$-subgroup $Q$ of $\G^F$ contained in $U$ there exists a unique $p$-block $B_Q$ of $\c_{\G^F}(Q)$ such that $(Q,B_Q)\leq (U,B_U)$ and, by applying \cite[Proposition 6.3.6]{Lin18II}, it follows that $(Q,B_Q)$ belongs to $\mathcal{S}^\star_p(B)$. We can therefore define a map of posets from $\mathcal{S}_p^\star(\G^F)$ to $\mathcal{S}_p^\star(B)$ by sending the $p$-subgroup $Q$ to the $B$-Brauer pair $(Q,B_Q)$. Observe that each Brauer pair $(Q,B_Q)$ is uniquely determined by $Q$ and $B$ and thus the map defined above is an isomorphism of posets. Furthermore, since $\G^F$ fixes $B$ and $\zeta$, we conclude that this map is $\G^F$-equivariant. We can now conclude that the induced map of simplicial complexes is a $\G^F$-homotopy equivalence between $\Delta(\mathcal{S}_p^\star(\G^F))$ and $\Delta(\mathcal{S}_p^\star(B))$ as claimed above. This concludes the proof according to the previous paragraph.
\end{proof}

\section{The non-defining characteristic case $\ell\neq p$}
\label{sec:Non-defining characteristic}

We keep $\G$, $F$, $p$ and $q$ as in Section \ref{sec:Defining characteristic} and assume now that $p\neq \ell$. We define $e_\ell(q)$ to be the multiplicative order of $q$ modulo $\ell$. The aim of this section is to prove Theorem \ref{thm:Main, Homotopy equivalence in non-defining characteristic} and obtain a description of the homotopy type of the simplicial complex $\Delta(\mathcal{S}_\ell^\star(B))$ for each $\ell$-block $B$ of $\G^F$ in terms of $e_\ell(q)$-Harish-Chandra theory. This result also extends \cite[Theorem A]{Ros-Homotopy} according to Lemma \ref{lem:Principal block and Brown complex}.

To start, we recall the definition of the set $\pi(\G,F)$ from \cite[Definition 2.1]{Ros-Homotopy}. Let $\G_{\rm sc}:=([\G,\G])_{\rm sc}$ be the group introduced in \cite[Example 1.5.3 (b)]{Gec-Mal20} and consider a pair $(\G^*,F^*)$ in duality with $(\G,F)$ as defined in \cite[Definition 1.5.17]{Gec-Mal20}. We refer the reader to \cite[2.7.14]{Gec-Mal20} for the definition of good primes.

\begin{defin}
\label{def:Prime condition}
Let $\pi'(\G,F)$ be the set of primes $\ell$ that are good for $\G$, do not divide $2$, $q$ or $|\z(\G_{\rm sc})^F|$, and satisfy $\ell\neq 3$ whenever $(\G,F)$ has a rational component of type ${^3{{\bf D}}_4}$. Then, we define $\pi(\G,F)$ to be the set of primes $\ell\in\pi'(\G,F)$ not diving $|\z(\G)^F:\z^\circ(\G)^F|$ nor $|\z(\G^*)^F:\z^\circ(\G^*)^F|$.
\end{defin}

Observe that if $\L$ is an $F$-stable Levi subgroup of $\G$ and $\ell$ is a prime in the set $\pi(\G,F)$, then it is not true that $\ell$ belongs to $\pi(\L,F)$. In fact, it might happen that $\ell$ divides the order of $\z(\L_{\rm sc})^F$. Nevertheless, in this case the prime $\ell$ is good for $\L$ and does not divide $|\z(\L)^F:\z^\circ(\L)^F|$ nor $|\z(\L^*)^F:\z^\circ(\L^*)^F|$. This observation will be used without further reference.

We will make use of the theory of $\Phi_e$-tori and $e$-split Levi subgroups as introduced in \cite{Bro-Mal92} where $e$ is a positive integer. In particular, for an $F$-stable torus $\T$, we denote by $\T_{\Phi_e}$ its Sylow $\Phi_e$-torus which is well-defined thanks to \cite[Theorem 3.4]{Bro-Mal92}. Then, we say that an $F$-stable Levi subgroup $\L$ of $\G$ is \textit{$e$-split} if it satisfies $\L=\c_\G(\z^\circ(\L)_{\Phi_e})$. In the next lemma we collect some well-known results on centralisers of abelian $\ell$-subgroups and their connection to $e$-split Levi subgroups.

\begin{lem}
\label{lem:Centralisers of abalian ell-subgroups}
Let $Q$ be an abelian $\ell$-subgroup of $\G^F$ and assume that $\ell$ is good for $\G$. Then:
\begin{enumerate}
\item $\H:=\c_\G^\circ(Q)$ is an $F$-stable Levi subgroup of $(\G,F)$;
\item $\L:=\c_\G(\z^\circ(\H)_{\Phi_e})$ is an $e$-split Levi subgroup of $(\G,F)$ and $\H\leq \L$;
\item if  $\ell$ does not divide $|\z(\G)^F:\z^\circ(\G)^F|$, then $\L=\c^\circ_\G(\z(\L)^F_\ell)$;
\item if $\ell$ does not divide $|\z(\G^*)^F:\z^\circ(\G^*)^F|$, then $\c_\G^\circ(Q)^F=\c_{\G^F}(Q)$;
\item if $\ell\in\pi(\G,F)$ and $e=e_\ell(q)$, then $\L=\G$ if and only if $Q\leq \z(\G)^F_\ell$; 
\end{enumerate}
\end{lem}

\begin{proof}
The statement in (i) follows from \cite[Proposition 13.16 (ii)]{Cab-Eng04} while (ii) is an immediate consequence of the definition of $e$-split Levi subgroup. For (iii) and (iv) see \cite[Proposition 13.19]{Cab-Eng04} and \cite[Proposition 13.16 (i)]{Cab-Eng04} respectively. Finally (v) follows from \cite[Lemma 3.10]{Ros-Homotopy}.
\end{proof}

Next, we recall the notion of connected subpairs and of twisted block induction introduced in \cite[Section 2]{Cab-Eng99}. We say that $(U,b_U)^\circ$ is a \textit{connected subpair} of $(\G,F)$ if $U$ is an abelian $\ell$-subgroup of $\G^F$ such that $U\leq \c_\G^\circ(U)$ and $b_U$ is a Brauer $\ell$-block of $\c_\G^\circ(U)^F$. If $V$ is another abelian $\ell$-subgroup of $\G^F$ with $V\leq U$, then there exists a unique Brauer $\ell$-block $b_V$ of $\c_\G^\circ(V)^F$ such that $\Br_U(b_V)b_U\neq 0$ in which case we write $(V,b_V)^\circ\lhd (U,b_U)^\circ$. Observe that, because the index of $\c_\G^\circ(U)^F$ in $\c_{\G^F}(U)$ is a power of $\ell$, there exists a unique $\ell$-block $\wh{b}_U$ of $\c_{\G^F}(U)$ that covers $b_U$. The relation between the Brauer pair $(U,\wh{b}_U)$ and the connected subpair $(U,b_U)^\circ$ is described in \cite[Proposition 2.2]{Cab-Eng99} and will be used in what follows without further reference. Next, assume that $\ell$ is good for $\G$ and consider an $e_\ell(q)$-split Levi subgroup $\L$ of $(\G,F)$. For any $\ell$-block $b_\L$ of $\L^F$ it was shown in \cite[Theorem 2.5]{Cab-Eng99}, using Deligne--Lusztig induction and assuming that $\ell$ is a good prime, that $b_\L$ corresponds to a unique $\ell$-block $b_\G$ of $\G^F$. This uniquely defined $\ell$-block of $\G^F$ is denoted by $b_\G=\R_\L^\G(b_\L)$ (see \cite[Notation 2.6]{Cab-Eng99}). Furthermore, whenever $\L=\c_\G^\circ(\z(\L)^F_\ell)$, we have the inclusion of connected subpairs $(1,\R_\L^\G(b_\L))^\circ\lhd(\z(\L)^F_\ell,b_\L)^\circ$.

Using the notion of twisted block induction we can now introduce an analogue of the simplicial complex $\Delta(\mathcal{S}_\ell^\star(B))$ adapted to finite reductive groups. These two complexes will then be shown to be homotopy equivalent.

\begin{defin}
Assume that $\ell$ is good for $\G$ and let $e=e_\ell(q)$. For every $\ell$-block $B$ of $\G^F$ we define the set $\CL_e^\star(B)$ of pairs $(\L,b_\L)$ where $\L$ is an $e$-split Levi subgroup of $(\G,F)$ such that $\L<\G$ and $b_\L$ is an $\ell$-block of $\L^F$ satisfying $\R_\L^\G(b_\L)=B$.
\end{defin}

Notice that the set $\CL^\star_e(B)$ defined above can be empty. However, under suitable assumptions on $\ell$, this cannot happen if the block $B$ has positive defect (see \cite[Lemma 4.8]{Ros-Homotopy}). Observe also that $\G^F$-acts by conjugation on the set $\CL_e^\star(B)$ and consider the order relation on $\CL_e^\star(B)$ given by $(\L,b_\L)\leq (\K,b_\K)$ if and only if $\R_\L^\K(b_\L)=b_\K$, for every $(\L,b_\L)$ and $(\K,b_\K)$ belonging to $\CL_e^\star(B)$. Then, we can construct a $\G^F$-simplicial complex $\Delta(\CL_e^\star(B))$ as explained in Section \ref{sec:Preliminaries}. We point out that the order relation $\leq$ defined above is closely related to the one introduced in \cite[Notation 1.11]{Cab-Eng99} (see also \cite[Section 3 and Section 4]{Ros-Generalized_HC_theory_for_Dade}).

Our aim is now to show that the simplicial complex of Brauer pairs $\Delta(\mathcal{S}_\ell^\star(B))$ is homotopy equivalent to $\Delta(\CL_e^\star(B))$. First we construct a suitable underlying map of posets. By Proposition \ref{prop:Almost centric pairs and homotopy equivalence} we can restrict our attention to the poset of almost-centric abelian Brauer pairs. 

\begin{prop}
\label{prop:Poset map for almost-centric Brauer pairs and e-split pairs}
Suppose that $\ell\in\pi(\G,F)$ does not divide the order of $\z(\G)^F$. For every $\ell$-block $B$ of $\G^F$ with non-trivial defect there exists a map of $\G^F$-posets
\[\phi:\ab_\ell^\star(B)^{\rm ac}\to\CL_e^\star(B)^{\rm op}\]
given by sending a Brauer pair $(Q,b_Q)$ to the pair $(\L,b_\L)$ where $\L=\c_\G(\z^\circ(\c_\G^\circ(Q))_{\Phi_e})$ and the block $b_\L$ is determined by the inclusion $(\z(\L)^F_\ell,b_\L)\leq (Q,b_Q)$.
\end{prop}

\begin{proof}
Fix $(Q,b_Q)\in\ab_\ell^\star(B)^{\rm ac}$ and define $\H:=\c_\G^\circ(Q)$ and $\L:=\c_\G(\z^\circ(\H)_{\Phi_e})$. By Lemma \ref{lem:Centralisers of abalian ell-subgroups} (i-ii) we know that $\H$ is an $F$-stable Levi subgroup and $\L$ an $e$-split Levi subgroup of $(\G,F)$ with $\H\leq \L$. In particular, it follows from the definition of Levi subgroup that $\z(\L)^F_\ell\leq \z(\H)^F_\ell$. On the other hand, $\z(\H)^F_\ell=\z(\H^F)_\ell$ while $\H^F=\c_{\G^F}(Q)$ by Lemma \ref{lem:Centralisers of abalian ell-subgroups} (iv). Now, using the fact that $Q$ is almost-centric, we have $\z(\H)^F_\ell=\z(\c_{\G^F}(Q))_\ell=Q$ and therefore we conclude that $\z(\L)^F_\ell\leq Q$. Then, according to \cite[Theorem 6.3.3]{Lin18II} there exists a unique block $b_\L$ of $\c_{\G^F}(\z(\L)^F_\ell)$ such that $(\z(\L)^F_\ell,b_\L)\leq (Q,b_Q)$. Observe that $\c_{\G^F}(\z(\L)^F_\ell)=\L^F$ by Lemma \ref{lem:Centralisers of abalian ell-subgroups} (iii-iv) and so $b_\L$ is a block of $\L^F$. We now define
\[\phi(Q,b_Q):=(\L,b_\L)\]
and claim that $\phi$ is a well-defined map of $\G^F$-posets. To start, we check that the pair $(\L,b_\L)$ actually belongs to $\CL_e^\star(B)$. Since $Q$ is non-trivial and $\ell$ does not divide the order of $\z(\G)^F$, it follows from Lemma \ref{lem:Centralisers of abalian ell-subgroups} (v) that $\L<\G$. Next, once again using the fact that $\L=\c_\G^\circ(\z(\L)^F_\ell)$, we can apply \cite[Theorem 2.5]{Cab-Eng99} to obtain the inclusion of connected subpairs $(1,\R_\L^\G(b_\L))^\circ\lhd (\z(\L)^F_\ell,b_\L)^\circ$. The latter is equivalent to the inclusion of Brauer pairs $(1,\R_\L^\G(b_\L))\leq(\z(\L)^F_\ell,b_\L)$ according to \cite[Proposition 2.2 (v)]{Cab-Eng99} because $\ell$ does not divide $|\z(\G^*)^F:\z^\circ(\G^*)^F|$. Then, we get
\[\left(1,\R_\L^\G(b_\L)\right)\leq \left(\z(\L)^F_\ell,b_\L\right)\leq \left(Q,b_Q\right)\]
and therefore that $\R_\L^\G(b_\L)=B$ because $(Q,b_Q)$ is a $B$-Brauer pair and by the uniqueness part of \cite[Theorem 6.3.3]{Lin18II}. This shows that $(\z(\L)^F_\ell,b_\L)$ belongs to the set $\CL_e^\star(B)$. Furthermore, it follows from the construction of $(\L,b_\L)=\phi(Q,b_Q)$ that for every $g\in \G^F$ we have $\L^g=\c_\G(\z^\circ(\c_\G^\circ(Q^g))_{\Phi_e})$ and $(\L^g,b_\L^g)\leq (Q^g,b_Q^g)$. In other words, we have $\phi((Q,b_Q)^g)=\phi(Q,b_Q)^g$. To conclude, let $(P,b_P)\in\ab_\ell^\star(B)^{\rm ac}$ satisfying $(Q,b_Q)\leq (P,b_P)$ and set $(\K,b_\K):=\phi(P,b_P)$. We want to show that $(\K,b_\K)\leq (\L,b_\L)$. By the discussion above, we already know that $\z(\L)^F_\ell\leq Q$ and similarly that $\z(\K)^F_\ell\leq P$. Moreover, using the fact that $Q\leq P$, we deduce that $\c_\G^\circ(P)\leq \c^\circ_\G(Q)$ and therefore that $\K\leq \L$ and $\z(\L)^F_\ell\leq \z(\K)^F_\ell$.
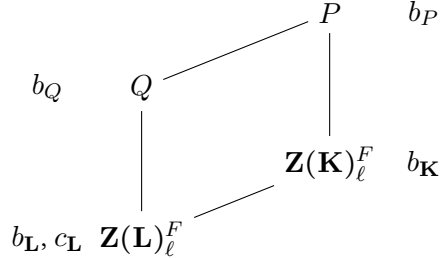
\begin{center}
\begin{tikzpicture}[x=2.5cm, y=1cm]
\node(N) at (0,0) {$\z(\L)^F_\ell$};
\node(K) at (0,2) {$Q$};
\node(Nx) at (1,1) {$\z(\K)^F_\ell$};
\node(Kx) at (1,3) {$P$};

\node(a) at (-0.5,0) {$b_\L$, $c_\L$};
\node(b) at (-0.5,2) {$b_Q$};
\node(c) at (1.5,1) {$b_\K$};
\node(d) at (1.5,3) {$b_P$};

\path[-]

(N) edge node {} (K)
(N) edge node {} (Nx)
(K) edge node {} (Kx)
(Nx) edge node {} (Kx);
\end{tikzpicture}
\end{center}
By the definition of $\phi$ we know that $(\z(\L)^F_\ell,b_\L)\leq (Q,b_Q)$ and $(\z(\K)^F_\ell,b_\K)\leq (P,b_P)$. Let now $c_\L$ be the unique block satisfying $(\z(\L)^F_\ell,c_\L)\leq (\z(\K)^F_\ell,b_\K)$ and observe that it satisfies $(\z(\L)^F_\ell,c_\L)\leq (P,b_P)$. On the other hand, since $(Q,b_Q)\leq (P,b_P)$ we get $(\z(\L)^F_\ell,b_\L)\leq (P,b_P)$ and the uniqueness of \cite[Theorem 6.3.3]{Lin18II} implies that $c_\L=b_\L$. This shows that $(\z(\L)^F_\ell,b_\L)\leq (\z(\K)^F_\ell,b_\K)$ holds inside $(\G,F)$ which is true if and only if $(1,b_\L)\leq (\z(\K)^F_\ell,b_\K)$ holds inside $(\L,F)$ according to \cite[Proposition 2.2 (i)]{Cab-Eng99} and because $\ell$ does not divide $|\z(\G^*)^F:\z^\circ(\G^*)^F|$. Then, by applying \cite[Theorem 2.5]{Cab-Eng99} to the block $b_\K$ of the $e$-split Levi subgroup $\K$ of $(\L,F)$, we conclude that $b_\L=\R_\K^\L(b_\K)$. This finally shows that $(\K,b_\K)\leq (\L,b_\L)$ as wanted and the proof is now complete.
\end{proof}

Next, we show that the fibres of the map $\phi$ are contractible. Recall that, given a poset $\mathcal{X}$ and an element $x\in\mathcal{X}$, we denote by $\mathcal{X}_{\geq x}$ the subposet consisting of those $y\in\mathcal{X}$ such that $x\leq y$.

\begin{lem}
\label{lem:Contractibility of fibers}
Consider the setting of Proposition \ref{prop:Poset map for almost-centric Brauer pairs and e-split pairs}. For every pair $(\L,b_\L)$ of $\CL_e^\star(B)$, the simplicial complex $\Delta(\mathcal{X})$ induced by the fibre $\mathcal{X}:=\phi^{-1}(\CL_e^\star(B)_{\geq (\L,b_\L)}^{\rm op})$ is $\n_{\G^F}(\L,b_\L)$-contractible.
\end{lem}

\begin{proof}
We claim that the pair $(\z(\L)^F_\ell,b_\L)$ is the minimum of the poset $\mathcal{X}$. First, observe that $\L^F=\c_{\G^F}(\z(\L)^F_\ell)$ by Lemma \ref{lem:Centralisers of abalian ell-subgroups} (iii-iv) and therefore that $\z(\L)^F_\ell$ is an almost-centric abelian $\ell$-subgroup of $\G^F$ as defined in Section \ref{sec:Preliminaries}. Then, since $(\L,b_\L)$ coincides with $\phi(\z(\L)^F_\ell,b_\L)$, it follows that $(\z(\L)^F_\ell,b_\L)$ belongs to the fibre $\mathcal{X}$. Consider now a Brauer pair $(P,b_P)\in\mathcal{X}$ and set $(\K,b_\K):=\phi(P,b_P)$. By the definition of $\mathcal{X}$ we have $(\K,b_\K)\leq (\L,b_\L)$ and so $b_\L=\R_\K^\L(b_\K)$. Then, \cite[Theorem 2.5]{Cab-Eng99} shows that $(1,b_\L)^\circ\leq (\z(\K)^F_\ell,b_\K)^\circ$ holds in $(\L,F)$. By \cite[Proposition 2.2 (i) and (v)]{Cab-Eng99} the latter is equivalent to the inclusion of Brauer pairs $(\z(\L)^F_\ell,b_\L)\leq (\z(\K)^F_\ell,b_\K)$. However, by the definition of $\phi$, we know that $(\z(\K)^F_\ell,b_\K)\leq (P,b_P)$ and therefore we conclude that $(\z(\L)^F_\ell,b_\L)\leq (P,b_P)$. This shows that $(\z(\L)^F_\ell,b_\L)$ is the minimum of $\mathcal{X}$ and, since $(\z(\L)^F_\ell,b_\L)$ is also $\n_{\G^F}(\L,b_\L)$-invariant, the result follows by applying \cite[Corollary 1.3]{Ros-Homotopy}.
\end{proof}

We can finally prove Theorem \ref{thm:Main, Homotopy equivalence in non-defining characteristic}. This follows by combining Proposition \ref{prop:Almost centric pairs and homotopy equivalence}, Proposition \ref{prop:Poset map for almost-centric Brauer pairs and e-split pairs}, and Lemma \ref{lem:Contractibility of fibers}.

\begin{proof}[Proof of Theorem \ref{thm:Main, Homotopy equivalence in non-defining characteristic}]
Set $e=e_\ell(q)$. By applying Proposition \ref{prop:Almost centric pairs and homotopy equivalence} and recalling that opposite posets induce homeomorphic simplicial complexes, it is enough to show that $\Delta(\ab_\ell^\star(B)^{\rm ac})$ is $\G^F$-homotopy equivalent to $\Delta(\CL_e^\star(B)^{\rm op})$. For this purpose, consider the map $\phi$ of $\G^F$-posets constructed in Proposition \ref{prop:Poset map for almost-centric Brauer pairs and e-split pairs}. Then, the induced map $\Delta(\phi)$ of simplicial complexes is a $\G^F$-homotopy equivalence according to Quillen's Theorem A (see the formulation given in \cite[Lemma 1.1]{Ros-Homotopy}) whose hypotheses are satisfied thanks to Lemma \ref{lem:Contractibility of fibers}.
\end{proof}

We conclude the paper by observing that \cite[Theorem A]{Ros-Homotopy} is a direct consequence of our Theorem \ref{thm:Main, Homotopy equivalence in non-defining characteristic}. In fact, by Lemma \ref{lem:Principal block and Brown complex} we know that the Brown complex $\Delta(\mathcal{S}_\ell^\star(\G^F))$ is $\G^F$-homotopy equivalent to the simplicial complex $\Delta(\mathcal{S}_\ell^\star(B_0(\G^F)))$ where $B_0(\G^F)$ denotes the principal $\ell$-block of $\G^F$. On the other hand, if we denote by $\CL_{e_\ell(q)}^\star(\G,F)$ the poset of proper $e$-split Levi subgroups of $(\G,F)$ (see \cite[Section 1.4]{Ros-Homotopy}), then the assignment $\L\mapsto (\L,B_0(\L^F))$ defines an isomorphism between the $\G^F$-posets $\CL_{e_\ell(q)}^\star(\G,F)$ and $\CL_{e_\ell(q)}^\star(B_0(\G^F))$ thanks to Brauer's Third Main Theorem (see \cite[Theorem 6.3.14]{Lin18II}). Finally, Theorem \ref{thm:Main, Homotopy equivalence in non-defining characteristic} implies that $\Delta(\mathcal{S}_\ell^\star(B_0(\G^F)))$ is $\G^F$-homotopy equivalent to $\Delta(\CL_{e_\ell(q)}^\star(B_0(\G^F)))$ and we conclude that $\Delta(\mathcal{S}_\ell^\star(\G^F))$ is $\G^F$-homotopy equivalent to $\Delta(\CL_{e_\ell(q)}^\star(\G,F))$ as in \cite[Theorem A]{Ros-Homotopy}.

\bibliographystyle{alpha}
\bibliography{References}

\vspace{1cm}


Mathematisches Forschungsinstitut Oberwolfach
\\
Schwarzwaldstr. 9-11,
\\
77709 Oberwolfach-Walke,
\\
Germany

\textit{Email address:} \href{mailto:damiano.rossi.math@gmail.com}{damiano.rossi.math@gmail.com}

\end{document}